\documentclass[11pt]{article}
\usepackage{graphics}
\usepackage{amssymb}

\usepackage{mathrsfs}
\usepackage{makeidx}
\usepackage{color}
\usepackage{graphicx}
\usepackage{subfigure}
\usepackage{multicol}
\usepackage{multirow}
\usepackage{float}
\usepackage{booktabs}
\usepackage{epsfig}

\usepackage{epstopdf}
\usepackage{caption}

\usepackage{algorithm,algorithmic}

\usepackage{latexsym, cite, bm, amsthm, amsmath}
\usepackage{enumerate}
\usepackage{marginnote}

\usepackage{cleveref}

\usepackage{pifont}

\def \[{\begin{equation}}
	\def \]{\end{equation}}

\newtheorem{theorem}{Theorem}[section]
\newtheorem{corollary}{Corollary}[section]
\newtheorem{lemma}{Lemma}[section]

\newtheorem{remark}{Remark}[section]

\newtheorem{example}{Example}[section]
\numberwithin{equation}{section}

\def\wtd{\widetilde}

\def\bbI{\mathbb{I}}
\def\bbR{\mathbb{R}}

\DeclareMathOperator{\rmc}{c}
\DeclareMathOperator{\card}{card}
\DeclareMathOperator{\diag}{diag}
\DeclareMathOperator{\sign}{sign}
\DeclareMathOperator{\rank}{rank}
\DeclareMathOperator{\T}{T}

\title{Solutions for Underdetermined Generalized Absolute Value Equations}

\usepackage[marginal]{footmisc}
\usepackage{authblk}

\author[a]{Cairong Chen\thanks{Supported in part by HK GRF 12300322 while he was a visiting scholar at HKBU from September 2022 to August 2023 and by the Fujian Alliance of Mathematics (2023SXLMQN03) and the Natural Science Foundation of Fujian Province (2021J01661). Email address: cairongchen@fjnu.edu.cn.}}
\author[a]{Xuehua Li\thanks{Email address: 3222714384@qq.com.}}
\author[b]{Ren-Cang Li\thanks{Supported in part by NSF grant DMS-2407692. Email address:  rcli@uta.edu.}}
\affil[a]{School of Mathematics and Statistics, Fujian Normal University, Fuzhou, 350117, P.R. China}
\affil[b]{Department of Mathematics,
           University of Texas at Arlington,
           P.O. Box 19408,
           Arlington, TX 76019-0408, USA}

\begin{document}
\date{\today}
\maketitle

\begin{abstract}
An underdetermined generalized absolute value equation (GAVE) may have no solution, one solution, finitely many or infinitely many solutions.
 This paper is concerned with sufficient conditions that guarantee the existence of solutions to an underdetermined GAVE. Particularly, sufficient conditions are established for
 an underdetermined GAVE to have infinitely many solutions with no zero entry that possess a particular or any given sign pattern. Iterative methods are proposed
 for the case when the underdetermined GAVE does have a solution.
 Some existing results for square GAVE are also extended.


\end{abstract}

\section{Introduction}\label{sec:intro}
We consider the generalized absolute value equation (GAVE) of  type
\begin{equation}\label{eq:gave}
Ax - B|x| = b,
\end{equation}
where $A, B\in \mathbb{R}^{m\times n}$ and $b\in \mathbb{R}^m$ are given,  $x\in\mathbb{R}^n$ is unknown, and $|x|$ stands for taking entrywise absolute value. By being underdetermined, we mean $m<n$, i.e., there are fewer equations in \eqref{eq:gave} than unknowns.
Square GAVE, i.e., $m=n$ in \eqref{eq:gave}, was first introduced by Rohn \cite{rohn2004}
and further investigated in \cite{mang2007a,prok2009,mezz2020,hlad2018} and references therein.

When $B=0$, GAVE~\eqref{eq:gave} reduces to the usual system of linear equations $Ax=b$,
which has been subjected to intensive studies \cite{saad2003} and plays a central role in scientific computing and engineering applications. When $m=n$ and $B$ is nonsingular,  GAVE~\eqref{eq:gave} can be converted into the so-called
 absolute value equation (AVE)
\begin{equation}\label{eq:ave}
Ax - |x| = b.
\end{equation}
Due to the appearance of $|x|$, solving  GAVE~\eqref{eq:gave} in general is NP-hard \cite{mang2007a}, and, if it has a solution, checking whether GAVE~\eqref{eq:gave} has a unique solution or multiple solutions is NP-complete \cite{prok2009}.

Over the past two decades, GAVE and AVE have received considerable attention
in the community of numerical optimization and numerical algebra.
GAVE and AVE are closely related to many important optimization problems,
such as the linear complementarity problem (LCP) and the mixed-integer programming \cite{mang2007a,mame2006,prok2009}.
In addition, AVE also arises from the characterization of certain solutions
to the system of linear interval equations \cite{rohn1989} and from the characterization
of the regularity of interval matrices \cite{rohn2009b}.
Recently, a novel transform function based on underdetermined GAVE is used
to improve the security of the cancellable biometric system \cite{dnhk2023}.
Conceivably in the years ahead, GAVE and AVE may find more applications in different contexts
where the underlying mathematical problem may be reformulated involving the absolute value operator.

There have been numerous investigations on both theoretical and numerical aspects for GAVE and AVE.
Theoretically, conditions for the existence,  nonexistence and uniqueness of solutions to GAVE~\eqref{eq:gave} or AVE~\eqref{eq:ave} can be found in, e.g.,  \cite{hlad2018,hlad2023,love2013,mame2006,mang2009b,mezz2020,rohn2009a,rohf2014,wugu2016,
wush2021,wuli2018,prok2009,
kdhm2024,hlmo2023,tzcy2023,zhwe2009,tazh2019,guo2024,chyh2023} and references therein. In particular, the connection between AVE~\eqref{eq:ave} and LCP is investigated in \cite{mame2006,huhu2010}.
Numerically, several algorithms for solving GAVE or AVE are available, such as, Newton-type iterative methods \cite{guo2024,mang2009a,shzh2023,wacc2019,zhwl2021,bcfp2016,tzcy2023,zhwe2009,lilw2018,
caqz2011,zahl2023},
iterative methods based on matrix splitting \cite{doss2020,kema2017,lild2022,zhzl2023,soso2023,edhs2017,chyh2024,alpa2023},
concave minimization approaches \cite{mang2007a,zahl2021,mang2007b},
methods based on dynamic models \cite{cyyh2021,maer2018,lyyh2023,yzch2024},
and others \cite{yuch2022,ke2020,rohf2014,tazh2019,jizh2013,chyh2023}.
An up-to-date overview can be found in \cite{hmhk2024}.
However, most of these existing results are about square GAVE, i.e., $m=n$ in \eqref{eq:gave}. There
are very few studies on non-square GAVE, i.e., $m\ne n$ in~\eqref{eq:gave}.
Mangasarian \cite{mang2007a} developed a successive linearization algorithm via concave minimization to solve non-square GAVE~\eqref{eq:gave}. Prokopyev \cite{prok2009} discussed connections of non-square GAVE~\eqref{eq:gave} with mixed integer programming and LCP. Rohn \cite{rohn2019} considered existence, uniqueness and computation of a solution for overdetermined GAVE, i.e., $m>n$ in \eqref{eq:gave}.
Recently, Xie, Qi and Han~\cite[Theorem~3.2]{xiqh2024} showed that a non-square GAVE~\eqref{eq:gave} cannot have a unique solution for any right-hand side $b\in \mathbb{R}^m$. Besides the numerical methods proposed in \cite{rohn2019,xiqh2024,mang2007a}, to the best of our knowledge, the methods that can handle non-square GAVE~\eqref{eq:gave} are the  alternating projection method \cite{alct2023} and the proximal gradient multi-step ascent/decent method \cite{dawz2024}.

Non-square GAVE does appear in real-world applications. For instance, in the proof of \cite[Proposition~2.1]{prok2009}, the reduction from the partition problem resulted in an overdetermined GAVE. As mentioned earlier, the  novel transform function based on an underdetermined GAVE is used to improve the security of the cancellable biometric system \cite{dnhk2023}. Non-square GAVE can also appear as a constraint in the absolute value programming problem \cite{mang2007a}.

An underdetermined GAVE may have no solution, one solution, finitely many or infinitely many solutions. The main goal of this paper is to investigate existence of solutions to  underdetermined GAVE~\eqref{eq:gave}. Along the way, numerical methods will also be developed to solve  underdetermined GAVE
when it has a solution.

The rest of this paper is organized as follows. In Section \ref{sec:UD-LinSys}, a few preliminary results on
underdetermined linear system are established. In Section \ref{sec:sufficondi}, the existence of infinitely many solutions and the existence of at least one solution to underdetermined GAVE are studied, respectively, and some numerical algorithms are developed as well.
Conclusion remarks are made in Section \ref{sec:Conclusions}.

\textbf{Notation.} We will adopt the following convention:
\begin{itemize}
  \item $\mathbb{R}^{m\times n}$ is the set of all $m \times n$ real matrices,
        $\mathbb{R}^{m}= \mathbb{R}^{m\times 1}$, and $\bbR=\bbR^1$.
        $I_n\in\bbR^{n\times n}$ is the identity matrix or simply $I$ if its size is clear from the context.
        For $x\in \mathbb{R}^n$, $\diag(x)$ denotes the diagonal matrix whose $i$th diagonal entry
        is the $i$th entry of $x$.

  \item $A^{\T}$, $A^\dag$ and $\rank(A)$ denote the transpose, the Moore-Penrose inverse and the rank of $A$, respectively.

  \item For $X\in \mathbb{R}^{m\times n}$, $X_{(i,j)}$ refers to its $(i,j)$th entry, $X_{(:,i)}$ its $i$th column,
        $X_{(i,:)}$ its $i$th row of $X$, and $X_{(:,i:j)}$ is the submatrix of $X$ consisting of its row $i$ to row $j$.
        In general, given $\bbI_1\subseteq \{ 1,2,\ldots, m\}$ and $\bbI_2 \subseteq \{1,2,\ldots, n\}$,
        $X_{(\bbI_1, \bbI_2)}$ is the submatrix of $X$ consisting of entries $X_{(i,j)}$ for $i\in\bbI_1,\, j\in\bbI_2$.
        Finally,
        $|X| \in\mathbb{R}^{m\times n}$ means taking entrywise absolute value.
        Analogously, for $x\in \mathbb{R}^{ n}$, $x_{(i)}$ is the $i$th entry of $x$ and
        $x_{\bbI}$ is the subvector of $x$ consisting of entries $x_{(i)}$ for $i\in\bbI$.
        For an index set $\bbI\subseteq \{1,2,\ldots, n\}$, $\card(\bbI)$ denotes its cardinality, and  $\bbI^{\rmc}$ is its complement with respect to $\{1,2,\ldots,n\}$.

  \item Inequality $X \le Y$ (or $X < Y$) means $X_{(i,j)} \le Y_{(i,j)}$ (or $X_{(i,j)} < Y_{(i,j)}$) for all $(i,j)$, and similarly for $X>Y$ and $X\ge Y$. In particular, $X\ge 0$ means that $X$ is entrywise nonnegative.

  \item $\|x\|_p:=(\sum_j|x_j|^p)^{1/p}$ for $1\le p\le\infty$ is
        the vector $p$-norm of $x=[x_i]\in\bbR^n$ and $\|A\|_p$ for $1\le p\le\infty$ is the $\ell_p$-operator norm.

  \item For $A\in\mathbb{R}^{n\times n}$, $\rho(A)$ and $\det (A)$ denote its spectral radius and determinant, respectively.

  \item The sign of $r\in\bbR$ is defined by $\sign(r)=1$ if $r> 0$, $0$ if $r=0$ and $-1$ if $r<0$.

\end{itemize}

\section{Solutions for an underdetermined linear system}\label{sec:UD-LinSys}
In this section, we will establish a few preliminary results on
underdetermined linear system
\begin{equation}\label{eq:le}
Ax = b,
\end{equation}
where $A\in \mathbb{R}^{m\times n}$ with $m<n$.
It is well known that linear system \eqref{eq:le} will have infinitely many solutions for any $b\in \mathbb{R}^m$ whenever $m<n$ and $A$ has full row rank, i.e., $\rank(A)=m$.

In the following, we will explore conditions under which \eqref{eq:le}
has infinitely many solutions having certain sign pattern or any sign pattern. By that
vector $x\in\bbR^n$ has the sign pattern
$s\in\{-1,0,1\}^n\subset\bbR^n$, we mean $\sign(x_{(i)})=s_{(i)}$ for $1\le i\le n$. In particular, $\diag(s)x\ge 0$.

\begin{lemma}\label{lem:uleg}
Let $A\in \mathbb{R}^{m\times n}$ and $b\in \mathbb{R}^m$ with $m<n$.
Suppose that the linear system~\eqref{eq:le} is consistent and has a solution $x_*$, and
let $A_{(\bbI_1,\bbI_2)}$ be a nonsingular submatrix of $A$ with $\card(\bbI_1) = \card(\bbI_2) = \rank(A)$ such that $x_{*\bbI_2}$ has no zero entry.
\begin{enumerate}[{\rm (a)}]
  \item If not all entries of $x_{*\bbI_2^{\rmc}}$ are $0$, then the linear system \eqref{eq:le} has infinitely many solutions with the same sign pattern as $x_*$;

  \item The linear system~\eqref{eq:le} has infinitely
        many solutions $x$ such that $x_{\bbI_2}$ has the same sign pattern as $x_{*\bbI_2}$;

  \item If $x_{*\bbI_2^{\rmc}}=0$, then $x_*$ is unique in the sense that \eqref{eq:le} has no solution other than $x_*$ with the same sign pattern as $x_*$.
\end{enumerate}
\end{lemma}

\begin{proof}
Because the linear system $Ax=b$ is assumed consistent, there are
$m-\rank(A)$ equations in the linear system that are redundant and can be simply dropped without causing any effect to the solution set of the linear system. For that reason, there is no loss of generality to assume $\rank(A)=m$.
Next, if necessary, upon permutating the  columns of $A$, there is no loss of generality, either,
to assume $\bbI_1=\bbI_2=\bbI:=\{1,2,...,m\}$.

With the two harmless assumptions we just made,  we partition $A=[A_1,A_2]$ where $A_1\in\bbR^{m\times m}$ is nonsingular and $x_{*\bbI}$ has no zero entries.
Next we rewrite $Ax=b$ as
$$
A_1x_{\bbI}=b-A_2x_{\bbI^{\rmc}},
$$
where $\bbI^{\rmc}=\{m+1,\ldots,n\}$, the complement of $\bbI$ with respect to $\{1,2,\ldots,n\}$.
Relatively perturb each entry $x_{*i}$ of $x_{*\bbI^{\rmc}}$ to $x_{*i}(1+\eta_i)$ for $|\eta_i|<1$, denote
the resulting vector by $x_{\bbI^{\rmc}}(\eta)$ as a vector valued function in $\eta=[\eta_i]$, and consider
the following linear system
\begin{equation}\label{eq:uleg-pf-1}
A_1x_{\bbI}=b-A_2x_{\bbI^{\rmc}}(\eta).
\end{equation}
For any given $\eta$ with each $|\eta_i|<1$, \eqref{eq:uleg-pf-1} has a unique solution
$x_{\bbI}(\eta)=A_1^{-1}[b-A_2x_{\bbI^{\rmc}}(\eta)]$ which is continuous with respect to $\eta$.
At $\eta=0$, that solution is $x_{*\bbI}$ with no zero entries. Hence by  continuity,
there exists $0<\epsilon<1$ such that each entry of $x_{\bbI}(\eta)$
has the same sign as the corresponding entry of $x_{*\bbI}$ whenever all $|\eta_i|<\epsilon<1$.
The proof is completed upon noticing $\begin{bmatrix}
                                       x_{\bbI}(\eta) \\
                                       x_{\bbI^{\rmc}}(\eta)
                                     \end{bmatrix}$
for all $\eta$ such that all $|\eta_i|<\epsilon<1$ has the same sign pattern
as $\begin{bmatrix}
                                       x_{*\bbI} \\
                                       x_{*\bbI^{\rmc}}
                                     \end{bmatrix}$.
\end{proof}

The condition of \Cref{lem:uleg} implies that $x_*$ has no fewer than $\rank(A)$ nonzero entries, and also
$A$ has a $\rank(A)$-by-$\rank(A)$ nonsingular submatrix for which the corresponding subvector of
$x_*$ has no zero entries. The following example shows that these conditions cannot be removed in general.

\begin{example}
{\rm
Consider  \eqref{eq:le} with
$$
A = \begin{bmatrix} 2& 1 &1\\1&2&-1\end{bmatrix}
\quad \text{and}\quad
b = \begin{bmatrix}1\\
                   -1
    \end{bmatrix}.
$$
It can be seen that $\rank(A)=2$ and any of its $2$-by-$2$ submatrix are nonsingular. It can also be seen that
$x_* = [0,0,1]^{\T}$ is a solution. But the condition: ``$A$ has a nonsingular submatrix $A_{(\bbI_1,\bbI_2)}$ with $\card(\bbI_1) = \card(\bbI_2) = \rank(A)$ such that $x_{*\bbI_2}$ has no zero entry" of \Cref{lem:uleg}
is violated. Note that the general solution of this linear system is $x = [-x_2,x_2,x_2+1]^{\T}$
for any $x_2\in \mathbb{R}$. Hence this $Ax=b$ does not have any other solution
with the same sign pattern as $x_*$. In fact, $x_*$ is  the only nonnegative solution.
}
\end{example}

An immediately consequence of \Cref{lem:uleg} is that the following corollary.

\begin{corollary}\label{cor:uleg}
Let $A\in \mathbb{R}^{m\times n}$ and $b\in \mathbb{R}^m$ with $m<n$.
Suppose that \eqref{eq:le} is consistent and has a solution $x_*$
with no zero entry. Then \eqref{eq:le} has infinitely many solutions
with the same sign pattern as that of  $x_*$.
\end{corollary}

\begin{proof}
Any nonsingular $\rank(A)\times\rank(A)$ submatrix of $A$ satisfies the requirement of
$A_{(\bbI_1,\bbI_2)}$ in \Cref{lem:uleg}.
\end{proof}

In the next lemma, we will show the need of having
$A_{(\bbI_1,\bbI_2)}$ in \Cref{lem:uleg} can be removed, if $x_*$ has more than $\rank(A)$ nonzero entries.

\begin{lemma}\label{lem:uleg'}
Let $A\in \mathbb{R}^{m\times n}$ and $b\in \mathbb{R}^m$ with $m<n$. Suppose that \eqref{eq:le} is consistent and has a solution $x_*$. If $x_*$ has more than $\rank(A)$ nonzero entries, then \eqref{eq:le}  has infinitely many solutions with the same sign pattern as that $x_*$.
\end{lemma}
\begin{proof}
Let $n'$ be the number of the nonzero entries of $x_*$. By the assumption of the lemma, we know $n'>\rank(A)$.
Without loss of generality, we may assume the last $n-n'$ entries of $x_*$ are $0$.
Next we focus our attention on the possible solutions of \eqref{eq:le} with last $n-n'$ entries being $0$, i.e.,
those $x=\begin{bmatrix}
           x_{(1:n')} \\
           0_{n-n'}
         \end{bmatrix}$ which satisfy
\begin{equation}\label{eq::pf-1}
A_{(:,1:n')}x_{(1:n')}=b.
\end{equation}
We notice that \eqref{eq::pf-1} is consistent because it is satisfied by $x_{*(1:n')}$.
Since $\rank(A_{(:,1:n')})\le\rank(A)$, there are $m-\rank(A_{(:,1:n')})$ redundant equations among the $m$ equations in the linear system $A_{(:,1:n')}x_{(1:n')}=b$. Remove these $m-\rank(A_{(:,1:n')})$ redundant equations to yield a new equivalent linear system
\begin{equation}\label{eq::pf-2}
\wtd Ax_{(1:n')}=\tilde b
\end{equation}
in the sense that the two linear systems share the same solution set. This new equation
\eqref{eq::pf-2} is consistent  because it is satisfied by $x_{*(1:n')}$, which has no zero entry,
and $\wtd A\in\bbR^{m'\times n'}$
where $m'=\rank(A_{(:,1:n')})\le\rank(A)<n'$. Hence the conditions of \Cref{cor:uleg} are satisfied for the linear system \eqref{eq::pf-2}, and therefore it has infinitely many solutions with the same sign pattern as
that $x_{*(1:n')}$. Finally, \eqref{eq:le}
 has infinitely many solutions in the form $\begin{bmatrix}
           x_{(1:n')} \\
           0_{n-n'}
         \end{bmatrix}$ with the same sign pattern as
that of $\begin{bmatrix}
           x_{*(1:n')} \\
           0_{n-n'}
         \end{bmatrix}$.
\end{proof}

\section{Solutions for underdetermined GAVE}\label{sec:sufficondi}
In this section, we will focus on exploring conditions under which an underdetermined GAVE has a solution.
Naturally, it is expected that the right-hand side $b$ holds an important factor in
the existence of a solution and the number of solutions, as \Cref{rem:ns} below demonstrates.
Equally expected, when $A$ and $B$ have certain properties,
GAVE~\eqref{eq:gave} may always have one or more solutions for any given  right-hand side $b$.
Our investigation in this section is divided into two cases:  solution existence that
is independent of $b$ or not.

\begin{example}\label{rem:ns}
{\rm
Consider GAVE~\eqref{eq:gave} with
$$
A = \begin{bmatrix} -1& 0&0\\ -1& 0 & 0\end{bmatrix} \quad \text{and}\quad B = \begin{bmatrix} -2& -1&-1\\ -2& -1 &1\end{bmatrix}.
$$
The table below shows a few examples of $b$, along with the number of solutions to GAVE~\eqref{eq:gave} for the given $b$.
$$
\renewcommand{\arraystretch}{1.3}
\begin{tabular}{|c|c|c|c|c|}
  \hline
  $b$ & $\begin{bmatrix}
           0 \\
           0
         \end{bmatrix}$
      & $\begin{bmatrix}
           1 \\
           -1
         \end{bmatrix}$
      & $\begin{bmatrix}
           1 \\
           1
         \end{bmatrix}$
      & $\begin{bmatrix}
           1 \\
           -1
         \end{bmatrix}$ \\ \hline
  number of solution(s) &
    $1$ &    $2$ &     $\infty$ &  $0$  \\
  \hline
\end{tabular}
$$
}
\end{example}

\subsection{Existence of solutions, independent of the right-hand side}\label{subsec:inf}
In this subsection, we will give  sufficient conditions, independent of the right-hand side,
to  guarantee the existence of one or more solutions to  underdetermined GAVE~\eqref{eq:gave}.

\begin{theorem}\label{thm:mpq}
Let $A,B\in \mathbb{R}^{m\times n}$ with $m<n$ and $1\le p\le\infty$.  If  $A = M_A - N_A$ such that $\rank(M_A) = m$ and
\begin{equation}\label{con:mpq}
\|M_A^\dag N_A\|_p+ \|M_A^\dag B\|_p < 1,
\end{equation}
then for any given $b\in\mathbb{R}^m$
\begin{equation}\label{eq:nms}
x = M_A^\dag\left(N_A x + B|x|+b\right)
\end{equation}
has a unique solution and that solution also satisfies
GAVE~\eqref{eq:gave}.
\end{theorem}

\begin{proof}
Write \eqref{eq:nms} as $x=\phi(x)$ where $\phi(x):=M_A^\dag\left(N_A x + B|x|+b\right)$. For any $x,\,z\in\bbR^n$,
noticing $\|\,|x|-|z|\,\|_p\le\|x-z\|_p$, we have
\begin{align*}
\|\phi(x)-\phi(z)\|_p&=\|M_A^\dag N_A (x-z) + M_A^\dag B(|x|-|z|)\|_p \\
    &\le\|M_A^\dag N_A\|_p\|x-z\|_p + \|M_A^\dag B\|_p\|\,|x|-|z|\,\|_p \\
    &\le\big(\|M_A^\dag N_A\|_p+ \|M_A^\dag B\|_p\big)\|x-z\|_p.
\end{align*}
Hence $\phi$ is a strict contraction mapping in $\bbR^n$ with respect to $\|\cdot\|_p$ by \eqref{con:mpq}.
By the Banach fixed-point theorem, $x=\phi(x)$ has a unique solution $x_*\in\bbR^n$.

It follows from $x_*=\phi(x_*)$ that $M_Ax_*=M_A \phi(x_*)$, i.e.,
$$
M_Ax_*=M_A\big[M_A^\dag(N_A x_* + B|x_*|+b)\big],
$$
which, together with $M_AM_A^{\dag}=I_m$ due to $\rank(M_A)=m$, yield
$Ax_*- B|x_*| = b$,
as was to be shown.
\end{proof}

The trivial splitting $M_A = A$ and $N_A = 0$ in \Cref{thm:mpq} leads to

\begin{corollary}\label{thm:mp}
Let $A,B\in \mathbb{R}^{m\times n}$ with $m<n$ and  $1\le p\le\infty$.
If $A$ has full row rank and if
\begin{equation}\label{con:mp}
\|A^\dag B\|_p<1,
\end{equation}
then for any given $b\in\mathbb{R}^m$
\begin{equation}\label{eq:nab}
x=A^\dag(B|x|+b)
\end{equation}
has a unique solution and that solution also satisfies
GAVE~\eqref{eq:gave}.
\end{corollary}

\begin{remark}
{\rm
In \Cref{thm:mpq} (resp. \Cref{thm:mp}), it is proved that the nonlinear system \eqref{eq:nms} (resp. \eqref{eq:nab}) has a unique
solution, which also satisfies
GAVE~\eqref{eq:gave}, under the condition \eqref{con:mpq} (resp. \eqref{con:mp}). That does not mean that  GAVE~\eqref{eq:gave} has a unique solution. For example,  for GAVE~\eqref{eq:gave}
with
$$
A =\begin{bmatrix}
        1 & 0  & 0 \\
        0  & 1 & 0 \\
      \end{bmatrix}, \quad
B=
      \begin{bmatrix}
        -0.5 & -0.5 & 0 \\
        0    & 0    & -0.5 \\
      \end{bmatrix}, \quad
b =\begin{bmatrix}
               1.75 \\
               0 \end{bmatrix},
$$
we have $\|A^\dag B\|_2\approx0.7071 < 1$. It can be checked that  $x_*=[\frac{7}{6},0,0]^{\T}$ is a solution of \eqref{eq:nab} and so is a solution of GAVE~\eqref{eq:gave}. However,  $\tilde{x}=[1,-\frac{1}{2},1]^{\T}$ is also a solution of GAVE~\eqref{eq:gave}, just that $\tilde x$ is not a  solution to \eqref{eq:nab}.
}
\end{remark}

Under the conditions of either \Cref{thm:mp} or \Cref{thm:mpq}, it follows from \cite[Theorem~3.2]{xiqh2024} that there is at least one $b\in \mathbb{R}^m$ such that GAVE~\eqref{eq:gave} has multiple solutions. The following theorem reveals conditions under which GAVE~\eqref{eq:gave} has infinitely many solutions for any given $b\in \mathbb{R}^m$.

\begin{theorem}\label{thm:info}
Consider GAVE~\eqref{eq:gave} and suppose $m<n$ and $\rank(A)=m$.
Let $A_1\in\bbR^{m\times m}$ be a nonsingular submatrix  of $A$ and $B_1\in\bbR^{m\times m}$ be
the corresponding submatrix of $B$ in the sense that $A_1$ and $B_1$ share the same column indexes as submatrices of $A$ and $B$, respectively. If $\|A_1^{-1}B_1\|_p<1$, then for any given
 $b \in\mathbb{R}^m$, GAVE~\eqref{eq:gave} has infinitely many solutions .
\end{theorem}

\begin{proof}
Without loss of generality, we may assume $A = [A_1,A_2]$ and $B=[B_1,B_2]$ and rewrite
GAVE~\eqref{eq:gave} as
\begin{equation}\label{eq:nave}
A_1\tilde{x} - B_1|\tilde{x}| = b - A_2 \hat{x} + B_2|\hat{x}|,
\end{equation}
where, accordingly, $x = [\tilde{x}^{\T},\hat{x}^{\T}]^{\T} \in \mathbb{R}^n$ with $\tilde{x}\in \mathbb{R}^m$.
For any given $\hat{x}$, \eqref{eq:nave} is a square GAVE in $\tilde{x}$.
In the same way  as the proof of
\cite[Corollary 4.3]{lich2025}, we can show that, given $\hat{x}$, GAVE \eqref{eq:nave}
has a unique solution $\tilde{x}_*$, yielding
a solution $\left[(\tilde{x}_*)^{\T}, \hat{x}^{\T} \right]^{\T}$ to the original GAVE~\eqref{eq:gave}.
Since $\hat{x}$ is arbitrary,  GAVE~\eqref{eq:gave} has infinitely many solutions.
\end{proof}

\begin{example}\label{exam:inf}
{\rm
Consider GAVE \eqref{eq:gave} with
$$
A = \begin{bmatrix}
        3 & 1 & 1 \\
        1 & 3 & 0 \\
      \end{bmatrix},\quad
B = \begin{bmatrix}
        1  & 0 & 1 \\
        0  & 1 & 0 \\
      \end{bmatrix},\quad \text{and} \quad b =\begin{bmatrix} 3\\0\end{bmatrix}.
$$
We take $A_1=
\begin{bmatrix}
        3 & 1  \\
        1 & 3  \\
      \end{bmatrix}$ and $B_1=
\begin{bmatrix}
        1  & 0  \\
        0  & 1 \\
\end{bmatrix}$ for which $\|A_1^{-1}B_1\|_2 =0.5<1$. According to \Cref{thm:info}, the corresponding GAVE \eqref{eq:gave} has infinitely many solutions. Indeed,  the GAVE can be transformed into
\begin{equation}\label{exam:3}
\begin{bmatrix}
        3 & 1  \\
        1 & 3
      \end{bmatrix}
    \begin{bmatrix}
        x_1 \\
        x_2
      \end{bmatrix}-
\begin{bmatrix}
        1  & 0  \\
        0  & 1 \\
      \end{bmatrix}
    \begin{bmatrix}
        |x_1| \\
        |x_2|
    \end{bmatrix}
=
\begin{bmatrix}
        3-x_3+|x_3| \\
        0 \\
      \end{bmatrix},
\end{equation}
which has a unique solution $[x_1,x_2]^{\T}$ for any given $x_3$, yielding a solution  $[x_1,x_2,x_3]^{\T}$ to the full GAVE~\eqref{eq:gave}, for example,
$$
\renewcommand{\arraystretch}{1.2}
\begin{array}{|c|c|c|}
  \hline
  x_3 & \mbox{unique solution to GAVE~\eqref{exam:3}} & \mbox{solution to GAVE~\eqref{eq:gave}} \\ \hline
  1 & [\frac{12}{7},-\frac{3}{7}]^{\T} & [\frac{12}{7},-\frac{3}{7},1]^{\T} \\
  0 & [\frac{12}{7},-\frac{3}{7}]^{\T} & [\frac{12}{7},-\frac{3}{7},0]^{\T} \\
  -1 & [\frac{20}{7},-\frac{5}{7}]^{\T} & [\frac{20}{7},-\frac{5}{7}, -1]^{\T} \\
  \hline
\end{array}
$$
}
\end{example}

None of the particular solutions in \Cref{exam:inf} are nonnegative. In fact, GAVE~\eqref{eq:gave} with $A$, $B$, and $b$ there cannot have nonnegative solution because GAVE~\eqref{eq:gave} for $x\ge 0$ becomes $(A-B)x=b$, i.e.,
$$
\begin{bmatrix}
  2 & 1 & 0 \\
  1 & 2 & 0
\end{bmatrix}x=\begin{bmatrix} 3\\0\end{bmatrix}
$$
which cannot have any nonnegative solution. This brings out an interesting question
on solutions of given sign pattern.
In addition, Theorem \ref{thm:info} does not say  whether GAVE \eqref{eq:gave} has infinitely many solutions with the same sign pattern as $[\frac{12}{7},-\frac{3}{7},1]^{\T}$ or as $[\frac{20}{7},-\frac{5}{7}, -1]^{\T}$. Nonetheless, it follows from \Cref{cor:gaveinf} below that this GAVE \eqref{eq:gave} not only has infinitely many solutions with the same sign pattern as $[\frac{12}{7},-\frac{3}{7},1]^{\T}$   but also has infinitely many solutions with the same sign pattern as $[\frac{20}{7},-\frac{5}{7}, -1]^{\T}$.

\subsection{Existence of solutions, dependent of the right-hand side}\label{ssec:dep-b}
In the following, we will turn our attention to the existence of solutions with conditions that depend on $b$. As mentioned earlier in \Cref{rem:ns}, dependent of $b$,  underdetermined GAVE~\eqref{eq:gave} may have no solution, one solution,  finitely many or infinitely many solutions. In general, if a solution exists, it is NP-complete to check whether  GAVE~\eqref{eq:gave} has one solution or multiple solutions \cite{prok2009}.

Suppose that we are seeking a solution $x$ of given sign pattern $s\in\{-1,0,1\}^n$ to GAVE~\eqref{eq:gave}.
Then $y=\diag(s)\,x\ge 0$ becomes a nonnegative solution of the following linear system
\begin{equation}\label{eq:tgave}
[A\diag(s)- B]y = b.
\end{equation}
Consequently, as a result of \Cref{lem:uleg}, we get

\begin{theorem}\label{thm:inf2-2}
Let $A,B\in \mathbb{R}^{m\times n}$ with $m<n$ and $b\in \mathbb{R}^m$. Suppose that GAVE~\eqref{eq:gave} has a solution $x_*$ and let $s={\rm sign}(x_*)\in\{-1,0,1\}^n$. Suppose also that
$x_*$ has no fewer than $\rank(A\diag(s) - B)$ nonzero entries, and
let $[A\diag(s) - B]_{(\bbI_1,\bbI_2)}$ be a nonsingular submatrix of
$A\diag(s) - B$
 with $\card(\bbI_1) = \card(\bbI_2) = \rank(A\diag(s) - B)$ such that $x_{*\bbI_2}$ has no zero entry.
\begin{enumerate}[{\rm (a)}]
  \item If not all entries of $x_{*\bbI_2^{\rmc}}$ are $0$, then GAVE~\eqref{eq:gave} has infinitely many solutions with the same sign pattern as $x_*$;

  \item GAVE~\eqref{eq:gave} has infinitely
        many solutions $x$ such that $x_{\bbI_2}$ has the same sign pattern as $x_{*\bbI_2}$;

  \item If $x_{*\bbI_2^{\rmc}}=0$, then $x_*$ is  unique in the sense that GAVE~\eqref{eq:gave} has no solution other than $x_*$ with the same sign pattern as $x_*$.
\end{enumerate}
\end{theorem}

\begin{proof}
GAVE~\eqref{eq:gave} has a solution $x_*$ with $x_{*\bbI_2}$ having no zero entry  implies that \eqref{eq:tgave} has a solution $y_* = \diag(s)x_*\ge 0$ for which $y_{*\bbI_2}>0$. Hence, the results follow directly from \Cref{lem:uleg}.
\end{proof}

Now we use an example to explain \Cref{thm:inf2-2}.

\begin{example}
{\rm
Consider GAVE~\eqref{eq:gave} with
$$A=\begin{bmatrix}
        1 & 0   & 0 \\
        0 & 2   & 0  \\
    \end{bmatrix}, \quad B=
  \begin{bmatrix}
      2 & 0 & 0 \\
      0 & 1 & 0 \\
    \end{bmatrix}\quad \text{and} \quad
    b=
    \begin{bmatrix}
    -1\\
    1
    \end{bmatrix}.
$$
It can be verified that $x_*=[1,1,0]^{\T}$ is a solution with the sign pattern $s = [1,1,0]^{\T}$.
Note $A\diag(s)-B=\begin{bmatrix} -1& 0 &0\\0 &1 & 0\end{bmatrix}$, $\rank(A\diag(s)-B)=2$, $[A\diag(s)-B]_{(:,1:2)}$ is nonsingular, $x_{*(1:2)} > 0$ and $x_{*3} = 0$. Hence, it follows from \Cref{thm:inf2-2} (c) that $x_*=[1,1,0]^{\T}$ is the unique solution to this GAVE with the sign pattern $s= [1,1,0]^{\T}$.
Similarly, it can be checked that $[-\frac{1}{3},1,0]^{\T}$ is the unique solution with the sign pattern
$[-1,1,0]^{\T}$.

On the other hand,
for the sign pattern $s=[1,1,1]^{\T}$, we also have $A\diag(s)-B=\begin{bmatrix} -1& 0 &0\\0 &1 & 0\end{bmatrix}$, $\rank(A\diag(s)-B)=2$, and $[A\diag(s)-B]_{(:,1:2)}$ is nonsingular. Note $x_*=[1,1,1]^{\T}$ is a solution of GAVE~\eqref{eq:gave}. Hence, it follows from \Cref{thm:inf2-2} (a) that this GAVE has infinitely many solutions with the sign pattern $s=[1,1,1]^{\T}$. Indeed, $x=[1,1,x_3]^{\T}$ with any $x_3 > 0$ is a solution. Using \Cref{thm:inf2-2} (a), we can also see that this
GAVE has infinitely many solutions with the same sign pattern as $[1,1,-1]^{\T}$, $[-\frac{1}{3},1,1]^{\T}$ and $[-\frac{1}{3},1,-1]^{\T}$, respectively.

Moreover, we can prove that this GAVE \eqref{eq:gave} has no solution with any other sign patterns. For instance, for the sign pattern $s=[-1,-1,-1]^{\T}$, we have $A\diag(s)-B=\begin{bmatrix}-3 & 0 & 0\\0 & -3 & 0\\\end{bmatrix}$. Then for
$y=\begin{bmatrix}\frac{1}{2} \\1 \end{bmatrix}$, we have
$(A\diag(s)-B)^{\T} y=[-\frac{3}{2},  -3, 0 ]^{\T}\leq 0$
and $b^{\T} y=\frac{1}{2} > 0$. Now we cite Farkas' theorem \cite[p. 31]{mang1994} to conclude that the linear system~\eqref{eq:tgave} has no nonnegative solution, which implies that the GAVE has no solution with sign pattern $s=[-1,-1,-1]^{\T}$.
}
\end{example}

The following theorem follows directly from \Cref{lem:uleg'}.

\begin{theorem}\label{thm:gaveinf}
Let $A,B\in \mathbb{R}^{m\times n}$ with $m<n$ and $b\in \mathbb{R}^m$. If GAVE~\eqref{eq:gave} has a solution $x_*$ with sign pattern $s= {\rm sign}(x_*)\in\{-1,0,1\}^n$ such that $x_*$ has more than $\rank(A\diag(s) - B)$ nonzero entries, then GAVE~\eqref{eq:gave} has infinitely many solutions with the sign pattern $s$.
\end{theorem}

\begin{proof}
Since GAVE~\eqref{eq:gave} has a solution $x_*$ with sign pattern $s$ such that $x_*$ has more than $\rank(A\diag(s) - B)$ nonzero entries,  \eqref{eq:tgave} has a nonnegative solution $y_*$ with more than $\rank(A\diag(s) - B)$ nonzero entries. It follows from \Cref{lem:uleg'} that  \eqref{eq:tgave} has infinitely many nonnegative solutions with the same sign pattern as $y_*$, which in turn means that GAVE~\eqref{eq:gave} has infinitely many solutions with the sign pattern $s$.
\end{proof}

We use the following example to illustrate \Cref{thm:gaveinf}.

\begin{example}
{\rm
Consider GAVE~\eqref{eq:gave} with
$$
A=\begin{bmatrix}
        1 & 0   & 0 \\
        1 & 0   & 0  \\
    \end{bmatrix}, \quad B=
  \begin{bmatrix}
      -1 & 0 & 0 \\
      -1 & 0 & 0 \\
    \end{bmatrix}\quad \text{and} \quad
    b=
    \begin{bmatrix}
    1\\
    1
    \end{bmatrix}.
$$
It can be verified that $x_*=[\frac{1}{2},1,0]^{\T}$ is a solution with sign pattern $s=[1,1,0]^{\T}$.
Note  $A\diag(s)-B=\begin{bmatrix} 2& 0 &0\\2 &0 & 0\end{bmatrix}$ and $\rank(A\diag(s)-B)=1$. It follows from \Cref{thm:gaveinf} that GAVE \eqref{eq:gave} has infinitely many solutions with the sign pattern $s$. Indeed, $x=[\frac{1}{2},x_2,0]^{\T}$ with $x_2>0$ is a solution of
GAVE~\eqref{eq:gave} which possesses the sign pattern $s$.
}
\end{example}

Since $\rank(A\diag(s)-B) \le m$, the following corollary follows immediately from  \Cref{thm:gaveinf}.

\begin{corollary}\label{cor:gaveinf}
Let $A,B\in \mathbb{R}^{m\times n}$ with $m<n$ and  $b\in \mathbb{R}^m$. If GAVE~\eqref{eq:gave} has a solution with more than $m$ nonzero entries, then GAVE~\eqref{eq:gave} has infinitely many solutions with the same sign pattern as the known solution.
\end{corollary}

So far, we have established several sufficient conditions under which the underdetermined GAVE~\eqref{eq:gave} has infinitely many solutions, dependent of $b\in \mathbb{R}^m$. Most of these conditions involve the number of the nonzero entries of  solutions, which in general is not known a priori. According to \Cref{cor:uleg}, if GAVE~\eqref{eq:gave} has a solution $x_*$ with no zero entry,  GAVE~\eqref{eq:gave} will have infinitely many solutions with the same sign pattern as that of $x_*$. An interesting question is whether GAVE~\eqref{eq:gave} has infinitely many solutions with given sign pattern $s\in \{1,-1\}^n$ or any
sign pattern $s\in \{1,-1\}^n$.

\begin{theorem}\label{thm:zn1}
Let $A, B\in\mathbb{R}^{m\times n}$ with $m<n$, $b\in\mathbb{R}^m$, and $s\in\{-1,1\}^n$, and let
$B = M_B - N_B$ with $\rank(M_B) = m$.
\begin{itemize}
  \item [{\rm (a)}] Let $1\le p\le\infty$. If
\begin{equation}\label{zn-con1}
M_B^\dag b \le 0, ~ M_B^\dag [N_B+A\diag(s)] \ge 0,~
\|M_B^\dag [N_B+A\diag(s)]\|_p < 1,
\end{equation}
   then there exists a nonnegative solution $y_*$ of \eqref{eq:tgave}    such that $x_* = \diag(s)y_*$ is a solution of GAVE~\eqref{eq:gave}.

  \item [{\rm (b)}] If
\begin{equation}\label{zn-con2}
M_B^\dag b < 0,~
\|M_B^\dag [N_B+A\diag(s)]\|_\infty < \dfrac{\gamma}{2}~ \text{with}~ \gamma = \frac{\min\limits_i|(M_B^\dag b)_i|}{\max\limits_i|(M_B^\dag b)_i|},
\end{equation}
  then
GAVE~\eqref{eq:gave} has infinitely many solutions with the sign pattern $s$.
\end{itemize}

\end{theorem}

\begin{proof}
We first claim that the linear system
\begin{equation}\label{eq:lemn}
y=M_B^\dag\left\{ [N_B + A\diag(s)]y - b \right\}
\end{equation}
has a unique solution which is nonnegative if \eqref{zn-con1} holds. That \eqref{eq:lemn} has a unique solution $y_*$
follows again from  the Banach fixed-point theorem since
$\phi(y):=M_B^\dag\left\{ [N_B + A\diag(s)]y - b \right\}$ is a strict contraction mapping
under either $\|M_B^\dag [N_B+A\diag(s)]\|_p < 1$ in \eqref{zn-con1} or
$\|M_B^\dag [N_B+A\diag(s)]\|_\infty < {\gamma}/{2}\le 1/2$ in \eqref{zn-con2}.
Since $M_B M_B^\dag = I_m$ due to $\rank(M_B) = m$, we get
\begin{equation*}
M_B y_*=  [N_B + A\diag(s)]y_* - b,
\end{equation*}
which implies that $y_*$ is a solution of \eqref{eq:tgave}.
Finally, $x_* = \diag(s)y_*$ is a solution of GAVE~\eqref{eq:gave} if also $y_*\ge 0$ which we have not prove yet.

In order to claim the non-negativeness of $y_*$, we need to get into the detail of the standard fixed point
iteration that proves the existence of $y_*$:
$y^{(0)}=-M_B^\dag b$ and
\begin{equation}\label{iter:mb}
y^{(k+1)}=M_B^\dag\left\{\left[N_B + A\diag(s)\right]y^{(k)} - b\right\}\quad\mbox{for $k\ge 0$},
\end{equation}
which is convergent.

For item (a), all $y^{(k)} \ge 0$ because $M_B^\dag b \le 0$ and $M_B^\dag [N_B+A\diag(s)] \ge 0$.
Therefore $y_*=\lim_{k\to\infty}y^{(k)}\ge 0$.
For item (b), we will have to bound $\|y_* - y^{(0)}\|_\infty$. Let $y^{(-1)}=0$ for convenience.
We have
\begin{align*}
\|y^{(k+1)}-y^{(k)}\|_{\infty}
  &\le\|M_B^\dag [N_B+A\diag(s)]\|_\infty\|y^{(k)}-y^{(k-1)}\|_{\infty} \\
  &\le\|M_B^\dag [N_B+A\diag(s)]\|_\infty^2\|y^{(k-1)}-y^{(k-2)}\|_{\infty} \\
  &\le\|M_B^\dag [N_B+A\diag(s)]\|_\infty^{k+1}\|y^{(0)}-y^{(-1)}\|_{\infty} \\
  &=\|M_B^\dag [N_B+A\diag(s)]\|_\infty^{k+1}\|M_B^\dag b\|_\infty,
\end{align*}
and thus
\begin{align*}
  \|y_* - y^{(0)}\|_\infty & = \left\|\sum_{k=0}^{\infty} \left(y^{(k+1)} - y^{(k)}\right)\right\|_\infty \\
  &\le \sum_{k=0}^{\infty} \left\|y^{(k+1)} - y^{(k)}\right\|_\infty \\
  &\le \sum_{k=0}^{\infty} \|M_B^\dag [N_B+A\diag(s)]\|_\infty^{k+1}\|M_B^\dag b\|_\infty \\
  &<\sum_{k=0}^{\infty} \left(\frac {\gamma}2\right)^{k+1}\|M_B^\dag b\|_\infty \\
    &= \frac{\gamma}{2-\gamma}  \|M_B^\dag b\|_\infty \\
    &\leq \gamma\|M_B^\dag b\|_\infty\\
    &=\min\limits_i|(M_B^\dag b)_i|,
\end{align*}
from which and knowing $y^{(0)}=-M_B^\dag b > 0$ we get
$$
\|y_* + M_B^\dag b\|_\infty < \min\limits_i ~-(M_B^\dag b)_i,
$$
which implies
$$
-y_{*i} - (M_B^\dag b)_i \leq |y_{*i} + (M_B^\dag b)_i| < -\max\limits_i~ (M_B^\dag b)_i,~i=1,2,\ldots,n.
$$
Consequently,
$$
0\leq \left[\max\limits_i~(M_B^\dag b)_i\right] - (M_B^\dag b)_i <y_{*i},~i=1,2,\ldots,n,
$$
as was to be shown.
Given $s\in\{-1,1\}^n$, GAVE~\eqref{eq:gave} has a solution with the sign pattern $s$ if and only if the linear equation \eqref{eq:tgave} has a positive  solution. In addition, by \Cref{cor:uleg}, the linear system~\eqref{eq:tgave} has infinitely many positive solutions whenever it has one, which implies that GAVE~\eqref{eq:gave} has infinitely many solutions with the sign pattern $s$.
\end{proof}

Next we use an example to explain \Cref{thm:zn1}.
\begin{example}
{\rm
Consider GAVE~\eqref{eq:gave} with
$$A=\begin{bmatrix}     -0.5 & 0   & -0.1 \\       0  & 0.1 & -0.1 \\    \end{bmatrix}, \quad B=
  \begin{bmatrix}
     0 & 0.1 & 0 \\
      0 & 0 & 0 \\
    \end{bmatrix}\quad \text{and} \quad b=\begin{bmatrix} -0.5\\ -0.25\end{bmatrix}.$$
We take $M_B= \begin{bmatrix}    0.5 & 0.1 & 0.1 \\    0 & 0.1 & 0.1 \\
  \end{bmatrix}$. Then $\rank(M_B) = 2$ and $M_B^\dag b= [-0.5,-1.25,-1.25]^{\T}<0$.
For $s = [1,-1,1]^{\T}$, we have $M_B^\dag[N_B + A\diag(s)]=0$, $\|M_B^\dag[N_B + A\diag(s)]\|_\infty=0<1$.
Thus, \eqref{zn-con1} holds and we can verify that $y_*=[\frac{1}{2},\frac{5}{2},0]^{\T}$  is a nonnegative solution of \eqref{eq:tgave} such that $x_*=\diag(s)y_*=[\frac{1}{2},-\frac{5}{2},0]^{\T}$ is a solution of this GAVE.
Moreover, since $0=2\|M_B^\dag[N_B + A\diag(s)]\|_\infty < \gamma=0.4$, it follows from \Cref{thm:zn1}(b) that this GAVE has infinitely many solutions with the sign pattern $s= [1,-1,1]^{\T}$. Indeed, $x = [\frac{1}{2}, x_3 - \frac{5}{2}, x_3]^{\T}$ with $x_3\in(0,\frac{5}{2})$ is a solution of the GAVE. However, this GAVE has no solution with the sign pattern
$s=[-1,1,-1]^{\T}$ or $[-1,-1,-1]^{\T}$. Indeed, for $s=[-1,1,-1]^\top$, the linear system \eqref{eq:tgave} becomes
\begin{equation*}
  \begin{bmatrix}
    0.5 & -0.1 & 0.1 \\
    0 & 0.1 & 0.1 \\
  \end{bmatrix}
  y
  =
  \begin{bmatrix}
    -0.5 \\
    -0.25 \\
  \end{bmatrix},
\end{equation*}
which does not have any positive solution.  For $s=[-1,-1,-1]^\top$, the linear system \eqref{eq:tgave} can be transformed into
\begin{equation*}
  \begin{bmatrix}
    0.5 & 0 & 0 \\
    0 & -0.1 & 0.1 \\
  \end{bmatrix}
  y
  =
  \begin{bmatrix}
    -0.25 \\
    -0.25 \\
  \end{bmatrix},
\end{equation*}
which also does not have any positive solution.
}
\end{example}

Letting $M_B = B$ and $N_B=0$ in \Cref{thm:zn1} (b) yields the following corollary. It is interesting to note that
with $N_B=0$,
$\|M_B^\dag [N_B+A\diag(s)]\|_\infty =
\|M_B^\dag A\|_\infty$
has nothing to do with sign pattern $s\in\{1,-1\}^n$.

\begin{corollary}\label{cor:zn}
Let $A,B \in\mathbb{R}^{m\times n}$ and $b\in\mathbb{R}^m$, $\rank(B) = m$ and $m<n$.
If $B^\dag b < 0$ and $\|B^\dag A\|_\infty < \dfrac{\gamma}{2}$
where $\gamma = \frac{\min\limits_i|(B^\dag b)_i|}{\max\limits_i|(B^\dag b)_i|}$,
then GAVE~\eqref{eq:gave} has infinitely many solutions for any given sign pattern $s\in\{1,-1\}^n$
\end{corollary}

\begin{remark}\label{rem:ims}
{\rm
In \cite[Proposition 6]{mame2006}, the authors showed that AVE~\eqref{eq:ave} has a unique solution for any given sign pattern $s\in\{1,-1\}^n$ if $b<0$ and $\|A\|_\infty < \frac{\gamma}{2}$, where $\gamma = \frac{\min_i |b_i|}{\max_i |b_i|}$. This can be directly extended to square GAVE~\eqref{eq:gave} whenever $B$ is nonsingular. However, unlike the underdetermined case (recall \Cref{cor:zn}), \cite[Proposition~4.1]{hlad2023} says that there is no AVE~\eqref{eq:ave}  which possesses infinitely many solutions for every sign pattern $s\in\{1,-1\}^n$.
}
\end{remark}

Letting $s=[1,1,\ldots,1]^T\in\bbR^n$ in \Cref{thm:zn1} (a) yields the following corollary.
\begin{corollary}\label{thm:nonective}
Let $A,~B\in\mathbb{R}^{m\times n}$ with $m<n$ and $b\in\mathbb{R}^m$. Let $B = M_B - N_B$ such that
$\rank(M_B) =m$. If
$$
M_B^\dag b\leq0, ~
M_B^\dag (N_B + A) \geq 0, ~
\|M_B^\dag (N_B + A)\|_p <1,
$$
then GAVE~\eqref{eq:gave} has at least one nonnegative solution.
\end{corollary}

Simply with $M_B=B$ and $N_B=0$, \Cref{thm:nonective} leads to

\begin{corollary}\label{cor:nonective}
Let $A,~B\in\mathbb{R}^{m\times n}$ and $b\in\mathbb{R}^m$ with $\rank(B) = m$ and $m<n$. If  $B^\dag b\leq0$, $B^\dag A\geq0$ and $\|B^\dag A\|_p<1$, then GAVE~\eqref{eq:gave} has at least one nonnegative solution.
\end{corollary}

Similarly to the proof of \Cref{thm:zn1}, we can obtain the following theorem.

\begin{theorem}\label{thm:soa}
Let $A, B\in\mathbb{R}^{m\times n}$ with $m < n$, $b\in\mathbb{R}^m$, and $s\in\{-1, 1\}^n$.
Let $A = M_{ A} - N_{ A}$ such that ${\rm rank}(M_{A}) = m$.
\begin{itemize}
  \item [{\rm (a)}] Let $1\le p\le\infty$.  If
  \begin{equation*}
  {\rm diag}(s)M_{A}^\dag b \geq 0,~{\rm diag}(s)M_{ A}^\dag [N_{A}{\rm diag}(s) + B] \geq 0,~\|M_{A}^\dag[N_{A}{\rm diag}(s) + B]\|_p< 1,
  \end{equation*}
then there exists a nonnegative solution $y_*$ of \eqref{eq:tgave}   such that $x_* = {\rm diag}(s)y_*$ is a solution of  GAVE~\eqref{eq:gave}.

  \item [{\rm (b)}] If
\begin{equation*}
{\rm diag}(s)M_A^\dagger b > 0, ~ \|M_A^\dagger [N_A {\rm diag}(s) + B]\|_{\infty} < \frac{\gamma}{2}~\text{with}~\gamma = \frac{\min\limits_i |(M_A^\dagger b)_i|}{\max\limits_i |(M_A^\dagger b)_i|},
\end{equation*}
then GAVE~\eqref{eq:gave} has infinitely many solutions with the sign pattern $s$.
\end{itemize}
\end{theorem}

Letting $s=[1,1,\ldots,1]^T\in\bbR^n$ in \Cref{thm:soa} (a) yields the following corollary.
\begin{corollary}\label{thm:nonectiveA}
Let $A,~B\in\mathbb{R}^{m\times n}$ with $m<n$ and $b\in\mathbb{R}^m$. Let $A = M_A - N_A$ such that
$\rank(M_A) =m$. If
$$
M_A^\dag b\geq0, ~
M_A^\dag (N_A + B) \geq 0, \quad
\|M_A^\dag (N_A + B)\|_p <1,
$$
then GAVE~\eqref{eq:gave} has at least one nonnegative solution.
\end{corollary}

Simply with $M_A=A$ and $N_A=0$, \Cref{thm:nonectiveA} leads to

\begin{corollary}\label{cor:nonectiveAa}
Let $A,~B\in\mathbb{R}^{m\times n}$ and $b\in\mathbb{R}^m$ with $\rank(A) = m$ and $m<n$. If  $A^\dag b\geq 0$, $A^\dag B\geq0$ and $\|A^\dag B\|_p<1$, then GAVE~\eqref{eq:gave} has at least one nonnegative solution.
\end{corollary}

\begin{remark}
{\rm
Since any convex combination of two nonnegative solutions of  GAVE~\eqref{eq:gave} is still a nonnegative solution,  the number of the nonnegative solution of GAVE~\eqref{eq:gave} can only be zero, one, or infinity. It is noted that
GAVE~\eqref{eq:gave} may have more than one solutions even if it has a
unique nonnegative solution. In what follows, we will use seven specific GAVE~\eqref{eq:gave} to illustrate various
possibilities:
\begin{alignat}{3}
A&=\begin{bmatrix}
      -0.5 & 0 & 0 \\
      0 & -0.1 & -0.05 \\
    \end{bmatrix}, &\quad
    B&=
    \begin{bmatrix}
      0 & 0    & 0 \\
      0 & -0.1 & 0 \\
    \end{bmatrix},&\quad
    b&=\begin{bmatrix}
        -0.5\\
        -0.25
      \end{bmatrix}; \label{eq:4egs-1} \\
A&=\begin{bmatrix}
        0 & 0.01 & 0\\
        0 & 0    & 0
      \end{bmatrix}, &\quad
    B&=\begin{bmatrix}
        0 & 0& 0\\
        0.01 & 0& 0.01
      \end{bmatrix},&\quad
      b &=\begin{bmatrix}
           0.1\\
           0
         \end{bmatrix}; \label{eq:4egs-2} \\
A&=\begin{bmatrix}
        0.1 & 0.1 & 0\\
        0 & 0 & 0
      \end{bmatrix}, &\quad
   B&=\begin{bmatrix}
        0 & 0  & 0\\
        0 & 0  & -0.1
      \end{bmatrix},&\quad
      b &=\begin{bmatrix}
           0\\
           1
         \end{bmatrix}; \label{eq:4egs-3} \\
A &=\begin{bmatrix}
        -1 & 0 & 0 & 0\\
        -1 & 0 & 0 & 0\\
        0 & 0 & 0 & 1
      \end{bmatrix}, &\quad
 B&=\begin{bmatrix}
        -2 & -1  & -1 & 0\\
        -2 & -1  & 1 & 0\\
        0 & 0 & 0 & 0
      \end{bmatrix},&\quad
 b &=\begin{bmatrix}
      1\\
      -1\\
      0
    \end{bmatrix}; \label{eq:4egs-4} \\
A &=\begin{bmatrix}
        0.5 &   0 & 0.5\\
        0   & 0.5 & 0
      \end{bmatrix},&\quad
  B&=\begin{bmatrix}
        1 & 0  & 0\\
        0 & 1  & 0
      \end{bmatrix},&\quad
      b &=\begin{bmatrix}
           -1\\ -1
         \end{bmatrix}; \label{eq:4egs-5} \\
A &=\begin{bmatrix}
        0.5 &   0.5 & 0.5\\
        0.5   & 0.5 & 0.5
      \end{bmatrix},&\quad
 B&=\begin{bmatrix}
        1 & 1  & 1\\
        2 & 1  & 1
      \end{bmatrix},&\quad
 b &=\begin{bmatrix}
      0\\
      0
    \end{bmatrix}; \label{eq:4egs-6}\\
    A &=\begin{bmatrix}
        -1 &   0 & 0\\
        -1   & 0 & 0
      \end{bmatrix},&\quad
 B&=\begin{bmatrix}
        -2 & -1  & -1\\
        -2 & -1  & 1
      \end{bmatrix},&\quad
 b &=\begin{bmatrix}
      1\\
      -1
    \end{bmatrix}. \label{eq:4egs-7}
\end{alignat}

For \eqref{eq:4egs-1}, let $M_B =
    \begin{bmatrix}
      0.5 & 0.1 & 0.1 \\
      0 & 0.01 & 0.1 \\
    \end{bmatrix}$. Then $\rank(M_B)=2$ and
    $$
    M_B^\dag(N_B + A)\approx
    \begin{bmatrix}
      0 & 0.1727 & 0.0883 \\
      0 & 0.0407 & 0.0652 \\
      0 & 0.0959 & 0.4935\\
    \end{bmatrix}\geq0,
    $$
    $M_B^\dag b\approx [-0.4413,-0.3262,-2.4674]^{\T}<0$, and $\|M_B^\dag(N_B + A)\|_2\approx 0.5233<1.$ Hence, the conditions of \Cref{thm:nonective} are satisfied. By some algebra, the associated GAVE has infinitely many nonnegative solutions, i.e., $x=[1,x_2,5]^{\T}$ with  $x_2\geq 0$.

For \eqref{eq:4egs-2}, let $M_B =
    \begin{bmatrix}
      0 & -0.5 & 0 \\
      0.01 & 0 & 0.01 \\
    \end{bmatrix}$. Then $\rank(M_B)=2$ and
    $$
    M_B^\dag(N_B + A)=
    \begin{bmatrix}
      0 & 0 & 0 \\
      0 & 0.98 & 0 \\
      0 & 0 & 0\\
    \end{bmatrix}\geq0,
    $$
    $M_B^\dag b= [0,-0.2,0]^{\T}\leq0$, and $\|M_B^\dag(N_B + A)\|_2= 0.98<1.$ Hence, the conditions of \Cref{thm:nonective} are satisfied. By some algebra, the associated GAVE  has a unique
    nonnegative solution $x=[0,10,0]^{\T}$.

    For \eqref{eq:4egs-3}, let $M_B =
    \begin{bmatrix}
      -0.1 & -0.1 & 0 \\
      0 & 0 & -0.1 \\
    \end{bmatrix}$. Then  $\rank(M_B)=2$ and
    $$M_B^\dag(N_B + A)=
    \begin{bmatrix}
      0 & 0 & 0 \\
      0 & 0 & 0 \\
      0 & 0 & 0\\
    \end{bmatrix}\geq 0,$$
    $M_B^\dag b= [0,0,-10]^{\T}\leq0$, and $\|M_B^\dag(N_B + A)\|_2=0<1.$ Hence, the conditions of \Cref{thm:nonective} are satisfied. By some algebra, the associated GAVE~\eqref{eq:gave} has infinitely many solutions while its nonnegative solution is unique.

    For \eqref{eq:4egs-4}, let $M_B =
    \begin{bmatrix}
      -2 & -1  & -1 & 0\\
        -2 & -1  & 1 & 0\\
        0 & 0 & 0 & -1
    \end{bmatrix}$. Then $\rank(M_B)=3$ and
    $$M_B^\dag(N_B + A)=
    \begin{bmatrix}
      0.4 & 0 & 0 & 0 \\
      0.2 & 0 & 0 & 0\\
      0 & 0 & 0 & 0\\
        0 & 0 & 0 & 0\\
    \end{bmatrix}\geq 0,$$
    $M_B^\dag b= [0,0,-1, 0]^{\T}\leq 0$, and $\|M_B^\dag(N_B + A)\|_2\approx 0.4472 <1.$ Hence, the conditions of \Cref{thm:nonective} are satisfied. By some algebra, the associated GAVE has two solutions $x = [0,0, \pm 1,0]^{\T}$ with one being nonnegative.

    For \eqref{eq:4egs-5}, we have $\rank(B) =2$, $B^\dag b=[-1,-1,0]^{\T}\leq 0$, $B^\dag A=
      \begin{bmatrix}
      0.5 & 0   & 0.5 \\
      0   & 0.5 & 0 \\
      0   & 0   & 0 \\
      \end{bmatrix}\geq 0$ and $\|B^\dag A\|_2\approx 0.7071<1$. Thus, the conditions of \Cref{cor:nonective} are satisfied. By some algebra, the associated GAVE has infinitely many nonnegative solutions, i.e., $x=[x_1, 2, x_1-2]^{\T}$ with $x_1\in[2,+\infty)$.

     For \eqref{eq:4egs-6}, we have $\rank(B) =2$, $B^\dag b=[0,0,0]^{\T}\leq 0$, $B^\dag A=
      \begin{bmatrix}
      0 & 0   & 0 \\
      0.25   & 0.25 & 0.25 \\
      0.25   & 0.25   & 0.25 \\
      \end{bmatrix}\geq 0$ and $\|B^\dag A\|_2\approx 0.6124 < 1$. Thus, the conditions of \Cref{cor:nonective} are satisfied. It can be checked that the associated GAVE has a unique  solution, i.e., the nonnegative solution $x=[0, 0, 0]^{\T}$.

      For \eqref{eq:4egs-7}, we have $\rank(B) =2$, $B^\dag b=[0,0,-1]^{\T}\leq 0$, $B^\dag A=
      \begin{bmatrix}
      0.4 & 0   & 0 \\
      0.2   & 0 & 0 \\
      0  & 0  & 0
      \end{bmatrix}\geq 0$ and $\|B^\dag A\|_2\approx 0.4472 < 1$. Hence, the conditions of \Cref{cor:nonective} are satisfied and the associated GAVE has a unique nonnegative solution $x = [0,0,1]^{\T}$ while it totally has two solutions.
      }
\end{remark}

\begin{remark}\label{unisoluave}
{\rm In \cite[Proposition 5]{mame2006}, the authors showed that a nonnegative solution to AVE~\eqref{eq:ave} exists if
\begin{equation}\label{avens}
A\ge 0, \quad \|A\|_2<1 \quad \text{and}\quad b\le 0.
\end{equation}
Indeed, we can prove that the nonnegative solution of AVE~\eqref{eq:ave} is unique if \eqref{avens} holds, which was not mentioned in \cite[Proposition 5]{mame2006}. In addition, the result can be  extended to square GAVE~\eqref{eq:gave} when $B$ is nonsingular, i.e., the nonnegative solution of GAVE~\eqref{eq:gave} is unique if $B^{-1} b\leq 0$, $B^{-1} A\geq0$ and $\|B^{-1} A\|_p<1$. When $B$ is singular, the nonnegative solution of GAVE~\eqref{eq:gave} is unique if there exists a splitting  $B = M_B -N_B$ such that $M_B^{-1} b \leq 0$, $M_B^{-1}(N_B + A)\geq 0$ and $\|M_B^{-1}(N_B+A)\|_p<1$.
}
\end{remark}

\begin{remark}\label{thm:tmp}{\rm
Let $A,\, B \in\mathbb{R}^{m\times n}$ with $m<n$, and $b=0\in\mathbb{R}^m$.
\begin{enumerate}[{\rm (a)}]
  \item Suppose $A = B$. If either $A_{(i,:)} > 0$ or $A_{(i,:)}<0$ for all $i\in\{1,2,\ldots,m\}$, then
       any nonnegative $x\in\bbR^n$ is a solution to GAVE~\eqref{eq:gave}.

  \item Suppose $A = -B$.  If $A_{(i,:)} > 0$ or $A_{(i,:)}<0$ for all $i\in\{1,2,\ldots,m\}$, then
       any nonpositive $x\in\bbR^n$ is a solution to GAVE~\eqref{eq:gave}.
\end{enumerate}}
\end{remark}

We end this section by pointing out that GAVE~\eqref{eq:gave} with $|A|<B$ (or $B<A\le 0$) and $b =0\in\mathbb{R}^m$ has a unique solution, i.e., the  trivial solution. We note that this situation cannot occur in the case of  underdetermined homogeneous linear system \eqref{eq:le},  a characteristic that distinguishes GAVE~\eqref{eq:gave} from homogeneous linear system because an underdetermined homogeneous linear system always has infinitely many solutions.

\section{Conclusion}\label{sec:Conclusions}
Existence of solutions for an underdetermined generalized absolute value equation (GAVE) are investigated. Conditions under which an underdetermined GAVE has at least one solution or infinitely many solutions are established. We also investigate whether an underdetermined GAVE has infinitely many solutions with no zero entry that possess a particular or any given sign
pattern. As a by-product, numerical methods are constructed to compute a solution to an underdetermined GAVE if it has a solution.

\bibliographystyle{plain}
\bibliography{ref-gave}

\end{document}